\documentclass[12pt]{article}

\usepackage{amssymb}
\usepackage{amsfonts}
\usepackage{amsthm}
\usepackage{amsmath}
\usepackage{float}
\usepackage{bigints}
\usepackage{fullpage}

\newtheorem{Lemma}{Lemma}

\newtheorem{Theorem}{Theorem}

\newtheorem{Definition}{Definition}
\newtheorem{Problem}{Problem}
\newtheorem{Corollary}{Corollary}
\newtheorem{Remark}{Remark}

\numberwithin{Subcase}{Case}

\DeclareMathOperator{\vol}{vol}
\DeclareMathOperator{\Svol}{Svol}

\newcommand{\Eu}{\mathbb{E}}

\newcommand{\BB}{\mathbf{B}}
\renewcommand{\S}{\mathbb{S}}
\renewcommand{\P}{\mathcal{P}}
\renewcommand{\c}{\mathbf{c}}
\newcommand{\q}{\mathbf{q}}
\newcommand{\p}{\mathbf{p}}
\renewcommand{\ll}{\bar{\lambda}}
\renewcommand{\o}{\mathbf{o}}
\DeclareMathOperator{\arccot}{arccot}
\DeclareMathOperator{\lin}{lin}
\DeclareMathOperator{\conv}{conv}
\DeclareMathOperator{\perim}{perim}
\DeclareMathOperator{\aff}{aff}
\DeclareMathOperator{\area}{area}
\DeclareMathOperator{\bd}{bd}
\DeclareMathOperator{\dist}{dist}
\DeclareMathOperator{\dif}{d}
\DeclareMathOperator{\inter}{int}
\DeclareMathOperator{\relbd}{relbd}

\usepackage[pdftex]{graphicx}
\graphicspath{{C:/Users/Sam/Desktop/PaperFigures/}}
\DeclareGraphicsExtensions{.pdf}

\title{Density bounds for outer parallel domains of unit ball packings
}

\author{K\'{a}roly Bezdek\thanks{Partially supported by a Natural Sciences and
Engineering Research Council of Canada Discovery Grant.} and Zsolt L\'angi\thanks{Partially supported by the J\'anos Bolyai Research Scholarship of the Hungarian Academy of Sciences.}}

\begin{document}

\maketitle

\begin{abstract}
We give upper bounds for the density of unit ball packings relative to their outer parallel domains and discuss their connection to contact numbers. Also, packings of soft balls are introduced and upper bounds are given for the fraction of space covered by them.

\end{abstract}

\section{Introduction}
\subsection{Upper bounds for the density of unit ball packings relative to their outer parallel domains}

Let $\Eu^{d}$ denote the $d$-dimensional Euclidean space, $d\ge 2$. As usual, let $\lin(\cdot )$, $\aff(\cdot )$, $\conv(\cdot )$, $\vol_d(\cdot )$, $\omega_d$, $\Svol_{d-1}(\cdot )$, $\dist(\cdot , \cdot )$, $\|\cdot \|$, and $\o$ refer to the linear hull, the affine hull, the convex hull in $\Eu^{d}$, the $d$-dimensional Euclidean volume measure, the $d$-dimensional volume of a $d$-dimensional unit ball, the $(d-1)$-dimensional spherical volume measure, the distance function in $\Eu^{d}$, the standard Euclidean norm, and to the origin in $\Eu^{d}$.

A family of closed $d$-dimensional balls of radii $1$ with pairwise disjoint interiors in $\Eu^{d}$ is called a {\it unit ball packing} in $\Eu^{d}$. The (upper) {\it density} of a unit ball packing is defined by an appropriate limit (\cite{F}, \cite{R64}) and is, roughly speaking, the proportion of space covered by the unit balls of the packing at hand. The {\it sphere packing problem} asks for the densest packing of unit balls in $\Eu^{d}$. This includes the computation of the {\it packing density} $\delta_d$ of unit balls in $\Eu^{d}$, which is the supremum of the upper densities of all unit ball packings in $\Eu^{d}$. The sphere packing problem is a longstanding open question with an exciting recent progress. For an overview on the status of the relevant research we refer the interested reader to \cite{B13}, \cite{CoZh14}, and \cite{H12}. Next, we recall two theorems on unit sphere packings that naturally lead us to the first problem of this paper.

The {\it Voronoi cell} of a unit ball in a packing of unit balls in $\Eu^{d}$ is the set of points that are not farther away from the center of the given ball than from any other ball's center. As is well known, the Voronoi cells of a unit ball packing in  $\Eu^{d}$ form a tiling of $\Eu^{d}$. One of the most attractive results on the sphere packing problem was proved by C. A. Rogers \cite{R} in 1958. It was rediscovered by Baranovskii \cite {Ba} and extended to spherical and hyperbolic spaces by K. B\"or\"oczky \cite {Bo1}. It can be phrased as follows. Take a regular $d$-dimensional simplex of edge length $2$ in $\Eu^{d}$ and then draw a $d$-dimensional unit ball around each vertex of the simplex. Let $\sigma _d$ denote the ratio of the volume of the portion of the simplex covered by balls to the volume of the simplex. Now, take a Voronoi cell of a unit ball in a packing of unit balls in the $d$-dimensional Euclidean space $\Eu^{d}, d\ge 2$ and then take the intersection of the given Voronoi cell with the $d$-dimensional ball of radius $\sqrt{\frac{2d}{d+1}}$ concentric to the unit ball of the Voronoi cell. (We note that $\sqrt{\frac{2d}{d+1}}$ is the circumradius of the regular $d$-dimensional simplex of edge length $2$ in $\Eu^{d}$.) Then the volume of the truncated Voronoi cell is at least $\frac{\omega_d}{\sigma_d}$. In other words, the density of each unit ball in its truncated Voronoi cell is at most $\sigma_d$. In 2002, the first named author \cite{B02} has improved Rogers's upper bound on the density of each unit ball in an arbitrary unit ball packing of $\Eu^{d}$ relative to its truncated Voronoi cell, by replacing $\sigma_d$ with $\widehat\sigma_d<\sigma_d$ for all $d\ge 8$.

The above truncation of Voronoi cells with balls concentric to unit balls makes it natural to introduce the following functionals for unit ball packings.

\begin{Definition}\label{defn:main}
Let $\BB^d=\{ \mathbf{x}\in \Eu^{d}\ |\ \|\mathbf{x}\|\le 1\}$ denote the closed unit ball centered at the origin $\o$ of $\Eu^{d}, d\ge 2$ and let $\P^n:=\{\c_i+\BB^d\ |\ 1\le i\le n \ {\rm with}\ \|\c_j-\c_k\|\ge 2\ {\rm for}\  {\rm all}\ 1\le j< k\le n\}$ be an arbitrary packing of $n>1$ unit balls in $\Eu^{d}$. The part of space covered by the unit balls of $\P^n$ is labelled by $\mathbf{P}^n:=\bigcup_{i=1}^{n}(\c_i+\BB^d)$. Moreover, let
$C^n:=\{\c_i\ |\ 1\le i\le n\}$ stand for the set of centers of the unit balls in $\P^n$. Furthermore, for any $\lambda >0$ let $\mathbf{P}^n_{\lambda}:=\bigcup \{ \mathbf{x}+\lambda\BB^d\ |\ \mathbf{x}\in \mathbf{P}^n\}=\bigcup_{i=1}^{n}(\c_i+(1+\lambda)\BB^d)$ denote the outer parallel domain of $\mathbf{P}^n$ having outer radius $\lambda$. Finally, let $$\delta_d(n, \lambda):=\max_{\P^n}\frac{n\omega_d}{\vol_d(\mathbf{P}^n_{\lambda})}=\frac{n\omega_d}{\min_{\P^n}  \vol_d\left(\bigcup_{i=1}^{n}(\c_i+(1+\lambda)\BB^d) \right)}\ {\text and}\ \delta_d(\lambda):=\limsup_{n\to +\infty}\delta_d(n, \lambda).$$
\end{Definition}

Now, let $\P:=\{\c_i+\BB^d\ |\ i=1,2,\dots \ {\rm with}\ \|\c_j-\c_k\|\ge 2\ {\rm for}\  {\rm all}\ 1\le j< k\}$ be an arbitrary infinite packing of unit balls in $\Eu^{d}$. Recall that
$$\delta_d=\sup_{\P}\left(\limsup_{R\to +\infty}  \frac{\sum_{\c_i+\BB^d\subset R\BB^d}\vol_d(\c_i+\BB^d)}{\vol_d(R\BB^d)}\right).$$
Hence, it is rather easy to see that $\delta_d\le\delta_d(\lambda)$ holds for all $\lambda>0, d\ge 2$.
On the other hand, it was proved in \cite{Be02} that {\it $\delta_d=\delta_d(\lambda)$ for all $\lambda\ge 1$} leading to the classical sphere packing problem. Furthermore, the theorem of \cite{R} (\cite {Ba}, \cite {Bo1}) quoted above states that $\delta_d(n, \lambda)\le \sigma_d$ holds for all $n>1, d\ge 2$, and $\lambda\ge \sqrt{\frac{2d}{d+1}}-1$. It implies the inequality $\delta_d\le\delta_d(\lambda)\le\sup_{n}\delta_d(n, \lambda) \le \sigma_d$ for all $d\ge 2, \lambda\ge \sqrt{\frac{2d}{d+1}}-1$. This was improved further by the above quoted theorem of \cite{B02} stating that $\delta_d(n, \lambda)\le \widehat\sigma_d<\sigma_d$ holds for all $n>1$ and $\lambda\ge \sqrt{\frac{2d}{d+1}}-1$ provided that $d\ge 8$. It implies the inequality $\delta_d\le\delta_d(\lambda)\le\sup_{n}\delta_d(n, \lambda) \le \widehat\sigma_d<\sigma_d$ for all $d\ge 8, \lambda\ge \sqrt{\frac{2d}{d+1}}-1$. Of course, any improvement on the upper bounds for $\delta_d\le\delta_d(\lambda)$ with $\lambda\ge \sqrt{\frac{2d}{d+1}}-1$ would be of interest. However, in this paper we focus on the closely related question on upper bounding $\delta_d(\lambda)$ over the complementary interval $0<\lambda<\sqrt{\frac{2d}{d+1}}-1$ for $d\ge2$. Thus, we raise an asymptotic problem on unit ball packings, which is a volumetric question on truncations of Voronoi cells of unit ball packings with balls concentric to unit balls having radii $1+\lambda>1$ reasonable close to $1$ in $\Eu^{d}$. More exactly, we put forward the following question.

\begin{Problem}\label{core}
Determine (resp., estimate) $\delta_d(\lambda)$ for $d\ge 2$, $0<\lambda<\sqrt{\frac{2d}{d+1}}-1$.
\end{Problem}

Before stating our results on Problem~\ref{core}, we comment on its connection to contact graphs of unit ball packings, a connection that would be interesting to explore further.
First, we note that $\frac{2}{\sqrt{3}}-1\le \sqrt{\frac{2d}{d+1}}-1$ holds for all $d\ge 2$. Second, observe that as $\frac{2}{\sqrt{3}}$ is the circumradius of a regular triangle of side length $2$, therefore if $0<\lambda<\frac{2}{\sqrt{3}}-1$, then for any unit ball packing $\P^n$ no three of the closed balls in the family $\{\c_i+(1+\lambda)\BB^d\ |\ 1\le i\le n\}$ have a point in common. In other words, for any $\lambda$ with $0<\lambda < \frac{2}{\sqrt{3}}-1$ and for any unit ball packing $\P^n$, in the arrangement $\{\c_i+(1+\lambda)\BB^d\ |\ 1\le i\le n\}$ of closed balls of radii $1+\lambda$ only pairs of balls may overlap. Thus, computing
$\delta_d(n, \lambda)$, i.e., minimizing $\vol_d(\mathbf{P}^n_{\lambda})$ means maximizing the total volume of pairwise overlaps in the ball arrangement $\{\c_i+(1+\lambda)\BB^d\ |\ 1\le i\le n\}$ with the underlying packing $\P^n$. Intuition would suggest to achieve this by simply maximizing the number of touching pairs in the unit ball packing $\P^n$. Hence, Problem~\ref{core} becomes very close to the {\it contact number problem} of finite unit ball packings for $0<\lambda<\frac{2}{\sqrt{3}}-1$.
Recall that the latter problem asks for the largest number of touching pairs, i.e., contacts in a packing of $n$ unit balls in $\Eu^{d}$ for given $n>1$ and $d>1$. We refer the interested reader to \cite{B12}, \cite{BS13} for an overview on contact numbers. Here, we state the following observation.

\begin{Theorem}\label{contact numbers and isoperimetry}
{\it Let $n>1$ and $d>1$ be given. Then there exists $\lambda_{d, n}>0$ and a packing  $\widehat{\P}^n$ of $n$ unit balls in $\Eu^{d}$ possessing the largest contact number for the given $n$ such that for all $\lambda$ satisfying $0<\lambda< \lambda_{d, n}$, $\delta_d(n, \lambda)$ is generated by $\widehat{\P}^n$, i.e., $\vol_d(\mathbf{P}^n_{\lambda})\ge \vol_d(\widehat{\mathbf{P}}^n_{\lambda})$ holds for every packing $\P^n$ of $n$ unit balls in $\Eu^{d}$.}
\end{Theorem}

Blichfeldt's method \cite{Bl} (see also \cite{FTK}) applied to Problem~\ref{core} leads to the following upper bound on $\delta_d(\lambda)$.

\begin{Theorem}\label{Blichfeldt-type}
Let $d$ and $\lambda$ be chosen satisfying $\sqrt[d]{d}-1\le\lambda\le\sqrt{2}-1$. Then
\begin{equation}\label{eq:Blichfeldt}
\delta_d(\lambda)\le\sup_{n}\delta_d(n, \lambda)\le\frac{2d+4}{\left(2-(1+\lambda)^2\right)d+4}(1+\lambda)^{-d}\le\frac{d+2}{2}(1+\lambda)^{-d}\le1.
\end{equation}
\end{Theorem}

We note that Blichfeldt's upper bound $\frac{d+2}{2}2^{-\frac{d}{2}}$ for the packing density of unit balls in $\Eu^{d}$ can be obtained from the upper bound formula of Theorem~\ref{Blichfeldt-type} by making the substitution $\lambda=\sqrt{2}-1$.

\begin{Theorem}\label{planar}
Let $\lambda$ be chosen satisfying $0<\lambda<\frac{2}{\sqrt{3}}-1=0.1547\dots$ and let $\mathbf{H}$ be a regular hexagon circumscribed the unit disk $\BB^2$ centered at the origin $\o$ in $\Eu^{2}$. Then
$$\delta_2(\lambda)=\frac{\pi}{{\rm area}\left(\mathbf{H}\cap (1+\lambda)\BB^2\right)}.$$
\end{Theorem}

\begin{Definition}\label{simplexbound}
Let $\mathbf{T}^d:={\rm conv}\{\mathbf{t}_1, \mathbf{t}_2, \dots ,\mathbf{t}_{d+1}\}$ be a regular $d$-simplex of edge length $2$ in  $\Eu^{d}, d\ge 2$
and let $0<\lambda< \sqrt{\frac{2d}{d+1}}-1$.
Set
\[\sigma_d(\lambda) := \frac{(d+1)\vol_d\left(\mathbf{T}^d\cap(\mathbf{t}_1+\BB^d)\right) }{\vol_d\left(\mathbf{T}^d\cap\left(\cup_{i=1}^{d+1}\mathbf{t}_i+(1+\lambda)\BB^d\right)  \right) } < 1.
\]
\end{Definition}

An elementary computation yields that if $0 < \lambda < \frac{2}{\sqrt{3}}-1$, then
\[
\sigma_3(\lambda) = \frac{\pi-6 \phi_0}{\pi \lambda^3 + \left(3\pi-9\phi_0 \right) \lambda^2 + \left(3\pi-18\phi_0 + \right)\lambda + \pi - 6 \phi_0},
\]
where $\phi_0:= \arctan\frac{1}{\sqrt{2}} = 0.615479 \ldots$.

\begin{Theorem}\label{Rogers-type}
Let $0<\lambda<\frac{2}{\sqrt{3}}-1=0.1547\dots$. Set $\psi_0 := -\arctan \left(\sqrt{ \frac{2}{3}} \tan \left( 5 \phi_0 \right) \right)  = 0.052438\ldots$.
Then
\begin{equation}\label{eq:Rogers}
\delta_3(\lambda)\le\sup_{n}\delta_3(n, \lambda)\le \frac{\pi-6\psi_0}{\pi-6\psi_0 + (3\pi-18\psi_0)\lambda - 18\psi_0 \lambda^2 - (\pi+6\psi_0) \lambda^3} < \sigma_3(\lambda).
\end{equation}
\end{Theorem}

Finally, we note that the concept of $\delta_d(n, \lambda)$ is different from the notion of parametric density introduced by Wills etc. (\cite{Wi99}). Namely,
the outer parallel domain of a packing of $n$ unit balls in $\Eu^{d}$ considered in this paper is non-convex and it is different from the outer parallel domain of the convex hull of the center points of a packing of $n$ unit balls in $\Eu^{d}$, which is the convex container needed for the definition of parametric density. 

\subsection{Upper bounds for the density of soft ball packings}

So far, we have discussed upper bounds for the densities $\delta_d\le\delta_d(\lambda)\le\sup_{n}\delta_d(n, \lambda)$ of unit ball packings relative to their outer parallel domains having outer radius $\lambda$ in $\Eu^{d}$. So, it is natural to go even further and investigate unit ball packings and their outer parallel domains by upper bounding the largest fraction of $\Eu^{d}$ covered by outer parallel domains of unit ball packings having outer radius $\lambda$ in $\Eu^{d}$. This leads us to the {\it packing problem of soft balls} introduced as follows.

\begin{Definition}
Let $\P:=\{\c_i+\BB^d\ |\ i=1,2,\dots \ {\rm with}\ \|\c_j-\c_k\|\ge 2\ {\rm for}\ {\rm all}\ 1\le j< k\}$ be an arbitrary infinite packing of unit balls in $\Eu^{d}$. Moreover, for any $d\ge 2$  and $\lambda\ge 0$ let $\mathbf{P}_{\lambda}:=\bigcup_{i=1}^{+\infty}(\c_i+(1+\lambda)\BB^d)$ denote the outer parallel domain of $\mathbf{P}:=\bigcup_{i=1}^{+\infty}(\c_i+\BB^d)$ having outer radius $\lambda$. Furthermore, let
$$\bar{\delta}_d(\mathbf{P}_{\lambda}):=\limsup_{R\to +\infty}\frac{\vol_d(\mathbf{P}_{\lambda}\cap R\BB^d)}{\vol_d(R\BB^d) }$$
be the (upper) density of the outer parallel domain $\mathbf{P}_{\lambda}$ assigned to the unit ball packing $\P$ in $\Eu^{d}$. Finally, let
$$\bar{\delta}_d(\lambda):=\sup_{\P}\bar{\delta}_d(\mathbf{P}_{\lambda})$$
be the largest density of the outer parallel domains of unit ball packings having outer radius $\lambda$ in $\Eu^{d}$. Putting it somewhat differently, one could say that the family
$\{\c_i+(1+\lambda)\BB^d\ |\ i=1,2,\dots\}$ of closed balls of radii $1+\lambda$ is a packing of soft balls with penetrating constant $\lambda$ if $\P:=\{\c_i+\BB^d\ |\ i=1,2,\dots\}$
is a unit ball packing  of $\Eu^{d}$ in the usual sense. In particular, $\bar{\delta}_d(\mathbf{P}_{\lambda})$ is called the (upper) density of the soft ball packing $\{\c_i+(1+\lambda)\BB^d\ |\ i=1,2,\dots\}$ with
$\bar{\delta}_d(\lambda)$ standing for the largest density of packings of soft balls of radii $1+\lambda$ having penetrating constant $\lambda$.
\end{Definition}

\begin{Problem}\label{complementary}
Determine (resp., estimate) $\bar{\delta}_d(\lambda)$ for $d\ge 2$, $0\le\lambda<\sqrt{\frac{2d}{d+1}}-1$.
\end{Problem}

Rogers's method \cite{R} (see also \cite{R64}) applied to Problem~\ref{complementary} leads to the following upper bound on $\bar{\delta}_d(\lambda)$.

\begin{Theorem}\label{Rogers-type2}
Let $\mathbf{T}^d:={\rm conv}\{\mathbf{t}_1, \mathbf{t}_2, \dots ,\mathbf{t}_{d+1}\}$ be a regular $d$-simplex of edge length $2$ in  $\Eu^{d}, d\ge 2$ and let $0\le\lambda<\sqrt{\frac{2d}{d+1}}-1$. Then
$$\bar{\delta}_d(\lambda)\le \bar{\sigma}_d(\lambda) :=\frac{\vol_d\left(\mathbf{T}^d\cap\left(\cup_{i=1}^{d+1}\mathbf{t}_i+(1+\lambda)\BB^d\right)  \right) } { \vol_d(\mathbf{T}^d)}<1.$$
\end{Theorem}

Clearly, Rogers's upper bound $\sigma _d$ for the packing density of unit balls in $\Eu^{d}$ is included in the upper bound formula of Theorem~\ref{Rogers-type2}
namely, with $\sigma _d=\bar{\sigma}_d(0)$.

\begin{Corollary} $\bar{\delta}_2(\lambda)=\bar{\sigma}_2(\lambda)$ holds for all  $0 \le \lambda < \frac{2}{\sqrt{3}}-1$.
\end{Corollary}

For the following special case we improve our Rogers-type upper bound on $\bar{\delta}_3(\lambda)$.

\begin{Theorem}\label{Rogers-type3}
Let $0 \le \lambda < \frac{2}{\sqrt{3}}-1$. Then
\[
\bar{\delta}_3(\lambda) \leq \frac{ \left( 20 \sqrt{6} \phi_0 -4 \sqrt{6} \pi -10\pi \right) (1+\lambda)^3+18 \pi (1+\lambda)^2 -6\pi}{3\pi-15\phi_0 + 5\sqrt{2}}<\bar{\sigma}_3(\lambda) ,
\]
where $\phi_0 = \arctan\frac{1}{\sqrt{2}} = 0.615479 \ldots$.
\end{Theorem}

As a special case, Theorem~\ref{Rogers-type3} for $\lambda=0$ gives the upper bound $0.778425\dots $ for the density of unit ball packings in $\Eu^{3}$ proved earlier by the first named author in \cite{B00}. More generally, as $\delta_d\le \delta_d(\lambda)\bar{\delta}_d(\lambda)$ holds for all $d\ge 2$ and $ \lambda> 0$, therefore upper bounds on $\delta_d(\lambda)$ and $\bar{\delta}_d(\lambda)$ imply upper bounds for  $\delta_d$ in a straightforward way.

In the rest of the paper we prove the theorems stated.
For concluding remarks see the last section of this paper.

\section{Proof of Theorem~\ref{contact numbers and isoperimetry}}

First, we show that there exists $\lambda'_{d, n}>0$ such that for every $\lambda$ satisfying $0<\lambda< \lambda'_{d, n}$, $\delta_d(n, \lambda)$ is generated by a packing of $n$ unit balls in $\Eu^{d}$ possessing the largest contact number $c(n,d)$ for the given $n$. Our proof is indirect and starts by assuming that the claim is not true. Then there exists a
sequence $\lambda_1>\lambda_2 >\dots>\lambda_m>\dots>0$ of positive reals with $\lim_{m\to+\infty}\lambda_m=0$ such that  the unit ball packing $\P(\lambda_m):=\{\c_i(\lambda_m)+\BB^d\ |\ 1\le i\le n \ {\rm with}\ \|\c_j(\lambda_m)-\c_k(\lambda_m)\|\ge 2\ {\rm for}\  {\rm all}\ 1\le j< k\le n\}$ that generates $\delta_d(n, \lambda_m)$ has a contact number $c(\P(\lambda_m))$ satisfying
\begin{equation}\label{contact-1}
c(\P(\lambda_m))\le c(n,d)-1
\end{equation}
 for all $m=1,2,\dots$. Clearly, by assumption
\begin{equation}\label{contact-2}
\vol_d(\mathbf{P}^n_{\lambda_m})\ge \vol_d(\mathbf{P}(\lambda_m))
\end{equation}
must hold for every packing $\P^n=\{\c_i+\BB^d\ |\ 1\le i\le n \ {\rm with}\ \|\c_j-\c_k\|\ge 2\ {\rm for}\  {\rm all}\ 1\le j< k\le n\}$ of $n$ unit balls in $\Eu^{d}$
and for all $m=1,2,\dots$, where $$\mathbf{P}^n_{\lambda_m}=\bigcup_{i=1}^{n}(\c_i+(1+\lambda_m)\BB^d)\ {\rm and}\ \mathbf{P}(\lambda_m):=\bigcup_{i=1}^{n}(\c_i(\lambda_m)+(1+\lambda_m)\BB^d).$$
By choosing convergent subsequences if necessary, one may assume that
$\lim_{m\to +\infty}\c_i(\lambda_m)=\c'_i\in \Eu^{d}$ for all $1\le i\le n$. Clearly, $\P':=\{\c'_i+\BB^d\ |\ 1\le i\le n\}$ is a packing of $n$ unit balls in $\Eu^{d}$. Now, let $\P'':=\{\c''_i+\BB^d\ |\ 1\le i\le n\}$ be a packing of $n$ unit balls in $\Eu^{d}$ with maximum contact number $c(n,d)$. Finally, let $2+2\lambda'$ be the smallest distance between the centers of non-touching pairs of unit balls in the packings  $\P'$ and $\P''$. Thus,  if $0<\lambda_m<\lambda'$ and $m$ is sufficiently large, then the number of overlapping pairs in the ball arrangement $\{\c_i(\lambda_m)+(1+\lambda_m)\BB^d\ |\ 1\le i\le n\}$ is at most $c(n,d)$. On the other hand, the number of overlapping pairs in the ball arrangement $\{\c''_i+(1+\lambda_m)\BB^d\ |\ 1\le i\le n\}$ is $c(n,d)$. Hence, (\ref{contact-1}) implies in a straightforward way that $ \vol_d(\mathbf{P}(\lambda_m))>\vol_d(\bigcup_{i=1}^{n}(\c''_i+(1+\lambda_m)\BB^d))$, a contradiction to (\ref{contact-2}). This completes our proof on the existence of $\lambda'_{d, n}>0$.

Second, we turn to the proof of the existence of the packing  $\widehat{\P}^n$ of $n$ unit balls in $\Eu^{d}$ with the extremal property stated in Theorem~\ref{contact numbers and isoperimetry}. According to the first part of our proof for every $\lambda$ satisfying $0<\lambda< \lambda'_{d, n}$ there exist a packing $\P(\lambda):=\{\c_i(\lambda)+\BB^d\ |\ 1\le i\le n \ {\rm with}\ \|\c_j(\lambda)-\c_k(\lambda)\|\ge 2\ {\rm for}\  {\rm all}\ 1\le j< k\le n\}$ of $n$ unit balls in $\Eu^{d}$ with contact number $c(\P(\lambda))=c(n,d)$ such that

\begin{equation}\label{contact-3}
\vol_d(\mathbf{P}^n_{\lambda})\ge \vol_d(\mathbf{P}(\lambda))
\end{equation}
holds for every packing $\P^n=\{\c_i+\BB^d\ |\ 1\le i\le n \ {\rm with}\ \|\c_j-\c_k\|\ge 2\ {\rm for}\  {\rm all}\ 1\le j< k\le n\}$ of $n$ unit balls in $\Eu^{d}$, where $$\mathbf{P}^n_{\lambda}=\bigcup_{i=1}^{n}(\c_i+(1+\lambda)\BB^d)\ {\rm and}\ \mathbf{P}(\lambda):=\bigcup_{i=1}^{n}(\c_i(\lambda)+(1+\lambda)\BB^d).$$ Now, if we assume that $\widehat{\P}^n$ does not exist, then clearly we must have a sequence $\lambda_1>\lambda_2 >\dots>\lambda_m>\dots>0$ of positive reals with $\lim_{m\to+\infty}\lambda_m=0$ and with unit ball packings  $\P(\lambda_m):=\{\c_i(\lambda_m)+\BB^d\ |\ 1\le i\le n \ {\rm with}\ \|\c_j(\lambda_m)-\c_k(\lambda_m)\|\ge 2\ {\rm for}\  {\rm all}\ 1\le j< k\le n\}$ in $\Eu^{d}$ each with maximum contact number $c(\P(\lambda_m))=c(n,d)$ such that we have (\ref{contact-3}), i.e.,
\begin{equation}\label{contact-4}
\vol_d(\mathbf{P}^n_{\lambda_m})\ge \vol_d(\mathbf{P}(\lambda_m))
\end{equation}
for every packing $\P^n=\{\c_i+\BB^d\ |\ 1\le i\le n \ {\rm with}\ \|\c_j-\c_k\|\ge 2\ {\rm for}\  {\rm all}\ 1\le j< k\le n\}$ of $n$ unit balls in $\Eu^{d}$ and for all $m=1,2,\dots$. In particular, we must have
\begin{equation}\label{contact-5}
\vol_d(\mathbf{P}(\lambda_{M}))\ge \vol_d(\mathbf{P}(\lambda_m))
\end{equation}
for all positive integers $1\le m\le M$. Last but not least by the non-existence of $\widehat{\P}^n$ we may assume about the sequence of the unit ball packings  $\P(\lambda_m), m=1,2,\dots$ (resp., of volumes $\vol_d(\mathbf{P}(\lambda_m)), m=1,2,\dots$) that for every positive integer $N$ there exist $m''>m'\ge N$ with
 \begin{equation}\label{contact-6}
\vol_d(\mathbf{P}(\lambda_{m''}))> \vol_d(\mathbf{P}(\lambda_{m'})).
\end{equation}
Finally, let $2+2\lambda'_m$ be the smallest distance between the centers of non-touching pairs of unit balls in the packing $\P(\lambda_m)$, $m=1,2,\dots$. We claim that there exists a positive integer $N'$ such that
\begin{equation}\label{contact-7}
0<\lambda_m<\lambda'_m \ {\rm for \  all}\  m\ge N'.
\end{equation}
Indeed, otherwise there exists a subsequence $\lambda'_{m_i}$, $i=1,2,\dots$ with $\lambda_{m_i}\ge\lambda'_{m_i}>0$ for all $i=1,2,\dots$ and so, with $\lim_{i\to+\infty}\lambda'_{m_i}=0$ implying the existence of a packing of $n$ unit balls in $\Eu^{d}$ (via taking a convergent subsequence of the unit ball packings $\P(\lambda_{m_i})$, $i=1,2,\dots$ in $\Eu^{d}$) with contact number at least $c(n,d)+1$, a contradiction.

\noindent Thus, (\ref{contact-7}) and $c(\P(\lambda_m))$ $=c(d,n)$ imply in a straightforward way that $\vol_d(\mathbf{P}(\lambda_{m''}))= \vol_d(\mathbf{P}(\lambda_{m'}))$ must hold for all $m''>m'\ge N'$, a contradiction to (\ref{contact-6}). This completes our proof of Theorem~\ref{contact numbers and isoperimetry}.

\section{Proof of Theorem~\ref{Blichfeldt-type}}
\bigskip

For simplicity, we set $\ll := 1 + \lambda$ and use it for the rest of the paper. In the proof that follows we apply Blichfeldt's idea to $\P^n$ within the container $\bigcup_{i=1}^n (c_i + \ll  \BB^d)$ following the presentation of Blichfeldt's method in \cite{FTK}.

For $i=1,2,\ldots,n$, let $\c_i = \left( c_{i1},c_{i2},\ldots,c_{in} \right)$.
Clearly, if $i \neq j$, we have $\| \c_i-\c_j\|^2 \geq 4$, or equivalently, $\sum_{k=1}^d (c_{ik}-c_{jk})^2 \geq 4$.
Summing up for all possible pairs of different indices, we obtain
\[
2 n (n-1) = 4 \binom{n}{2} \leq n \sum_{i=1}^n \left( \sum_{j=1}^d c_{ij}^2 \right) - \sum_{j=1}^d \left( \sum_{i=1}^n c_{ij} \right)^2,
\]
which yields
\begin{equation}\label{eq:Blichfeldt2}
2(n-1)\leq \sum_{i=1}^n \|\c_i\|^2.
\end{equation}

We need the following definitions and lemma.

\begin{Definition}
The function
\[
\rho_\lambda(\mathbf{x}) = \left\{
\begin{array}{ll}
1-\frac{1}{2} \|\mathbf{x}\|^2, & \hbox{if } \|\mathbf{x}\| \leq \ll \\
0, & \hbox{if } \|\mathbf{x} \| > \ll
\end{array}
\right.
\]
is called the \emph{Blichfeldt gauge function}.
\end{Definition}

\begin{Lemma}\label{lem:Blichfeldt}
For any $\mathbf{y} \in \Eu^d$, we have
\[
\sum_{i=1}^n \rho_\lambda (\mathbf{y}-\c_i) \leq 1.
\]
\end{Lemma}

\begin{proof}
Without loss of generality, let $\mathbf{y}$ be the origin.
Then, from (\ref{eq:Blichfeldt2}), it follows that
\[
\sum_{i=1}^n \rho_\lambda (\c_i) = \sum_{\|\c_i\| \leq \ll} \left( 1- \frac{1}{2}\|\c_i\|^2 \right) \leq \sum_{i=1}^n \left( 1- \frac{1}{2}\|\c_i\|^2 \right) =
n-\frac{1}{2}\sum_{i=1}^n \|\c_i\|^2 \leq n - \frac{1}{2} 2 \cdot (n-1) = 1.
\]
\end{proof}

\begin{Definition}
Let
\[
I(\rho_\lambda) = \int_{\Eu^d} \rho_\lambda(\mathbf{x}) \dif \mathbf{x}, \quad \delta = \frac{n \omega_d}{\vol_d(\bigcup_{i=1}^n (\c_i + \ll \BB^d))}, \quad \Delta = \delta \frac{I(\rho_\lambda)}{\omega_d}.
\]
\end{Definition}

Clearly, Lemma~\ref{lem:Blichfeldt} implies that $\Delta \leq 1$, and therefore $\delta \leq \frac{\omega_d}{I(\rho_\lambda)}$, which yields that
$\delta_d(n,\lambda) \leq \frac{\omega_d}{I(\rho_\lambda)}$.

Now,
\[
I(\rho_\lambda) = \int_{\Eu^d} \rho_\lambda(\mathbf{x}) \dif \mathbf{x} = \int_{\ll \BB^d} \left( 1-\frac{1}{2}\|\mathbf{x}\|^2 \right) \dif \mathbf{x}=
\int_{0}^{\ll} \left( 1-\frac{1}{2}r^2 \right) r^{d-1} d \omega_d \dif r
\]
\[
= \omega_d \left( \ll^d - \frac{d}{2(d+2)} \ll^{d+2} \right).
\]

Hence, we have
\[
\delta_d(n,\lambda) \leq \frac{1}{ \ll^d \left( 1 - \frac{d}{2d+4} \ll^2 \right)  }
= \frac{2d+4}{\left( 2-\ll^2 \right)d + 4 } \ll^{-d} ,
\]
and the assertion follows.

\section{Proof of Theorem~\ref{planar}}
\bigskip

Let $\P^n = \left\{ \c_i+ \BB^2 : i=1,2,\ldots, n\right\}$ be a packing of $n$ unit disks in $\Eu^2$, and let $1 < \ll =1+\lambda< \frac{2}{\sqrt{3}}$.

\begin{Definition}\label{defn:lambdasimple}
The \emph{$\lambda$-intersection graph} of $\P^n$ is the graph $G(\P^n)$ with $\{ \c_i :i=1,2,\ldots,n \}$ as vertices,
and with two vertices connected by a line segment if their distance is at most $2\ll$.
\end{Definition}

Note that since $1 < \ll < \frac{2}{\sqrt{3}}$, the $\lambda$-intersection graph of $\P^n$ is planar, but if $\ll > \frac{2}{\sqrt{3}}$,
it is not necessarily so.

\begin{Definition}\label{defn:perimeter}
The unbounded face of the $\lambda$-intersection graph $G(\P^n)$ is bounded by finitely many closed sequences of edges of $G(\P^n)$.
We call the collection of these sequences the \emph{boundary of $G(\P^n)$}, and denote the sum of the lengths of the edges in them by $\perim (G(\P^n))$.
\end{Definition}

We remark that an edge of $G(\P^n)$ may appear more than once in the boundary of $G(\P^n)$
(for instance, if the boundary of the unbounded face contains a vertex of degree one).
Such an edge contributes its length more than once to $\perim (G(\P^n))$.

We prove the following, stronger statement, which readily implies Theorem~\ref{planar}.

\begin{Theorem}\label{thm:Groemer}
Let $\P^n = \left\{ \c_i+\BB^2 : i=1,2,\ldots,n \right\}$ be a packing of $n$ unit disks in $\Eu^2$, and let $1 < \ll < \frac{2}{\sqrt{3}}$.
Let $A = \area \left( \bigcup_{i=1}^{n}(\c_i+\ll\BB^2) \right)$ and $P=\perim (G(\P^n))$.
Then
\begin{equation}\label{eq:Groemer}
A \geq \left(  \area \left(\mathbf{H}\cap \ll \BB^2\right) \right) n + \left( \ll^2 \arccos\frac{1}{\ll} - \sqrt{\ll^2-1} \right) P + \ll^2\pi .
\end{equation}
\end{Theorem}

We note that Theorem~\ref{thm:Groemer} is a generalization of a result of Groemer in \cite{Groemer60}.


\begin{proof}
An elementary computation yields
\begin{equation}\label{eq:GroemerDVcell}
\area \left(\mathbf{H}\cap \ll\BB^2\right) = \ll^2 \left( \pi - 6\arccos \frac{1}{\ll} \right) + 6\sqrt{\ll^2 - 1}.
\end{equation}

Let $C$ denote the union of the bounded faces of the graph $G(\P^n)$.
Consider the Voronoi decomposition of $\Eu^2$ by $\P^n$.
Observe that as $\ll < \frac{2}{\sqrt{3}}$, no point of the plane belongs to more than two disks of the family $ \left\{ \c_i+ \ll\BB^2 : i=1,2,\ldots, n\right\}$.
Thus, if $E=[\c_i,\c_j]$ is an edge of $G(\P^n)$, the midpoint $\mathbf{m}$ of $E$ is a common point of the Voronoi cells of $\c_i + \BB^2$ and $\c_j + \BB^2$;
more specifically $\mathbf{m}$ is the point of the common edge of these cells, closest to both $\c_i$ and $\c_j$.
Hence, following Roger's method \cite{R64}, we may partition $C$ into triangles of the form $T=\conv \{ \c_i,\c_i',\c_i''\}$,
where $\c_i'$ is the point on an edge $E$ of the Voronoi cell of $\c_i + \BB^2$, closest to $\c_i$, and $\c_i''$ is an endpoint of $E$.
We call these triangles \emph{interior cells}, define the \emph{centre} of any such cell $T=\conv \{ \c_i,\c_i',\c_i''\}$ as $c_i$, and its
\emph{angle} as the angle $\angle(\c_i',\c_i,\c_i'')$.
Furthermore, we define the \emph{edge contribution} of an interior cell to be zero.

Now, let $[\c_i, \c_j]$ be an edge in the boundary of $G(\P^n)$, with outer unit normal vector $\mathbf{u}$ and midpoint $\mathbf{m}$.
Then the sets $\left( [\c_t, \mathbf{m}] + \left[ \o, \ll \mathbf{u} \right] \right) \cap \left( \c_t + \ll \BB^2 \right)$, where $t \in \{i,j \}$,
are called \emph{boundary cells}, with centre $c_t$ (Figure~\ref{fig:bdcell}). We define their angles $\frac{\pi}{2}$, and their edge contributions $\frac{1}{2} \| \c_i - \c_j \|$.
Note that, even though no two interior cells overlap, this is not necessarily true for boundary cells: such a cell may have some overlap with interior as well as boundary cells.

\begin{figure}
\begin{center}
\includegraphics[width=0.4\textwidth]{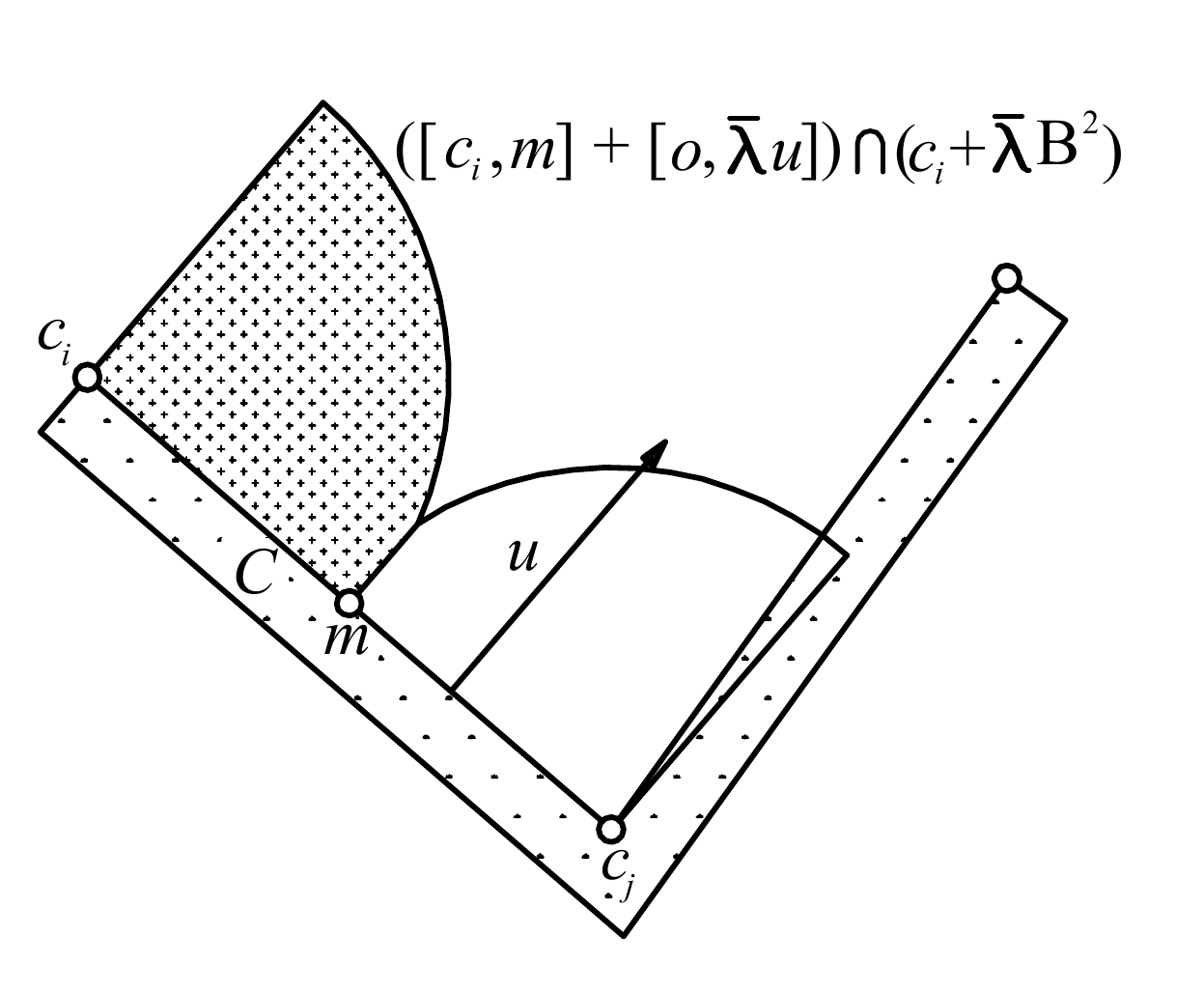}
\caption[]{Boundary cells: the one with centre $c_i$ is denoted by crosses, and $\inter C$ is represented by dots.}
\label{fig:bdcell}
\end{center}
\end{figure}

The proof of Theorem~\ref{thm:Groemer} is based on Lemma~\ref{lem:Groemer}.

\begin{Lemma}\label{lem:Groemer}
Let $T$ be an interior or boundary cell with centre $\c$, edge contribution $x$ and angle $\alpha$.
Then
\begin{equation}\label{eq:Groemerlemma}
\area\left( T \cap \left( \c + \ll\BB^2 \right) \right) \geq
\frac{\ll^2\left( \frac{\pi}{6} - \arccos\frac{1}{\ll} \right) + \sqrt{\ll^2 - 1}}{\frac{\pi}{3}} \alpha
+ \left( \ll^2 \arccos \frac{1}{\ll} - \sqrt{\ll^2-1} \right) x.
\end{equation}
\end{Lemma}

First, we show how Lemma~\ref{lem:Groemer} yields Theorem~\ref{thm:Groemer}.
Let the (interior and boundary) cells of $\P^n$ be $T_j$, $j=1,2,\ldots, k$, with centre $\c_j$, angle $\alpha_j$ and edge contribution $x_j$.
Let $T_j' = T_j \cap \left( \c_j + \ll \BB^2 \right)$.
Since the sum of the (signed) turning angles at the vertices of a simple polygon is equal to $2\pi$, we have
\[
A= \sum_{j=1}^k \area(T_j') + s \ll^2 \pi,
\]
where $s$ is the number of components of the boundary of $G(\P^n)$.
On the other hand,
\[
\sum_{j=1}^k \alpha_j = 2\pi n,\quad  \hbox{and} \quad \sum_{j=1}^k x_j = P.
\]
Thus, summing up both sides in Lemma~\ref{lem:Groemer}, and using the estimate $s \geq 1$ implies Theorem~\ref{thm:Groemer}.
\end{proof}

\begin{proof}[Proof of Lemma~\ref{lem:Groemer}]

For simplicity, let $T' = T \cap \left(\c+\ll\BB^2 \right)$.

First, we consider the case that $T$ is a boundary triangle.
Then $\alpha = \frac{\pi}{2}$, and an elementary computation yields that
\begin{equation}\label{eq:Groemerbd}
\area\left( T' \right) = \frac{\ll^2}{2} \left( \frac{\pi}{2} - \arccos\frac{x}{\ll} \right) + \frac{x}{2} \sqrt{\ll^2-x^2}.
\end{equation}
Combining (\ref{eq:Groemerlemma}) and (\ref{eq:Groemerbd}), it suffices to show that the function
\[
f(x)=-\frac{\ll^2}{2}\arccos\frac{x}{\ll}+\frac{x}{2}\sqrt{\ll^2-x^2} + \left( \frac{3}{2} - x \right)\left( \ll^2 \arccos \frac{1}{\ll} - \sqrt{\ll^2-1} \right)
\]
is not negative for any $1 \leq x \leq \ll \leq \frac{2}{\sqrt{3}}$.
Note that
\[
f''(x) = \frac{-x}{\sqrt{\ll^2-x^2}},
\]
and hence, $f$ is a strictly concave function of $x$, from which it follows that it is minimal either at $x=1$ or at $x=\ll$.

Now, we have $f(1) = 0$, and $f(\ll) = \left( \frac{3}{2} - \ll \right)\left( \ll^2 \arccos \frac{1}{\ll} - \sqrt{\ll^2-1} \right)$.
Since $\ll \leq \frac{2}{\sqrt{3}} < \frac{3}{2}$, the first factor of $f(\ll)$ is positive.
On the other hand, comparing the second factor to (\ref{eq:GroemerDVcell}), we can see that it is equal to
$\frac{1}{6}\area \left(   \ll \BB^2 \setminus \mathbf{H} \right) > 0$.

Second, let $T$ be an interior cell triangle, which yields that $x=0$.
Observe that if $T = \conv \{ \c,\mathbf{x},\mathbf{y} \}$ is not a right triangle, then
both $\mathbf{x}$ and $\mathbf{y}$ are vertices of the Voronoi cell of $\c + \BB^2$, from which it follows that $\| \mathbf{x}-\c \|, \|\mathbf{y}-\c \| \geq \frac{2}{\sqrt{3}}$.
In this case $T'$ is a circle sector, and $\area(T') = \ll^2 \frac{\alpha}{2}$, which yields the assertion.
Thus, we may assume that $T = \conv \{ \, \mathbf{x},\mathbf{y} \}$ has a right angle at $\mathbf{x}$, and that $\|\mathbf{x}-\c \| < \frac{2}{\sqrt{3}}$.
Moving $\mathbf{y}$ towards $\mathbf{x}$ increases the ratio $\frac{\area(T')}{\alpha}$, and hence, we may assume that
$\|\mathbf{y}-\c\|=\frac{2}{\sqrt{3}}$.
Under these conditions, we have
\[
\area (T') = \frac{\ll^2}{2} \left( \alpha - \arccos \frac{2 \cos \alpha}{\sqrt{3} \ll} \right) + \frac{1}{\sqrt{3}} \cos \alpha \sqrt{\ll^2 - \frac{4}{3}\cos^2 \alpha},
\]
and, combining it with (\ref{eq:Groemerlemma}), it suffices to show that the function
\[
g(\alpha) = - \frac{\ll^2}{2} \arccos \frac{2 \cos \alpha}{\sqrt{3} \ll}  + \frac{1}{\sqrt{3}} \cos \alpha \sqrt{\ll^2 - \frac{4}{3}\cos^2 \alpha} + \frac{\ll^2}{2} \arccos\frac{1}{\ll} \alpha - \frac{\alpha}{2} \sqrt{r^2-1}
\]
is not negative if $1 \leq \ll \leq \frac{2}{\sqrt{3}}$ and $\arccos\frac{\sqrt{3}\ll}{2} \leq \alpha \leq \frac{\pi}{6}$.
To do this, we may apply a computation similar to the one in case of a boundary triangle.
\end{proof}

\section{Proof of Theorem~\ref{Rogers-type}}\label{sec:proofRogers-type1}

First of all, recall that $\ll=1+\lambda$, and let
\[
\delta := \frac{\pi-6\psi_0}{\pi-6\psi_0 + (3\pi-18\psi_0)\lambda - 18\psi_0 \lambda^2 - (\pi+6\psi_0) \lambda^3} < \sigma_3(\lambda).
\]

Consider a unit ball packing $\P^n$ in $\Eu^3$, and let $V$ be the Voronoi cell of some ball of $\P^n$, say $\BB^3$.
Let $F$ be a face of $V$, and denote the intersection of the conic hull of $F$ with $V$, $\BB^3$ and $\bd \BB^3 = \S^2$ by $V_F$, $B_F$ and $S_F$, respectively.
Furthermore, we set $V'_F = V_F \cap \left( \ll \BB^3 \right)$.
To prove Theorem~\ref{Rogers-type}, it is sufficient to show that
\begin{equation}\label{eq:Rogers1}
\frac{\vol_3(B_F)}{\vol_3(V'_F)} \leq \delta.
\end{equation}

Recall the well-known fact (cf. \cite{R64}) that the distance of any $(d-i)$-dimensional face of $V$ from $o$ is at least $\sqrt{\frac{2i}{i+1}}$.
Thus, $\ll < \frac{2}{\sqrt{3}}$ yields that the intersection of $\aff F$ with $\ll\BB^3$ is either contained in $F$, or disjoint from it.
In the second case $\frac{\vol_3(B_F)}{\vol_3(V'_F)} = \frac{1}{\ll^3} < \delta$, and thus, we may assume that
$\aff F \cap \left( \ll\BB^3 \right) \subset F$.

Let the distance of $F$ and $\o$ be $x$, where $1 \leq x \leq \ll < \frac{2}{\sqrt{3}}$.
An elementary computation yields that $\vol_3\left( \left( \ll B_F \right) \setminus V_F \right) =
\pi \left( \frac{2}{3} \ll^3 - \ll^2 x + \frac{1}{3} x^3 \right)$, from which it follows that
\begin{equation}\label{eq:DVcellRogers}
\frac{\vol_3(V'_F)}{\vol_3(B_F)} = \ll^3 - \frac{\pi \left( \frac{2}{3} \ll^3 - \ll^2 x + \frac{1}{3} x^3 \right)}{\vol_3(B_F)}.
\end{equation}

First, we intend to minimize $\vol_3(B_F)$, while keeping the value of $x$ fixed.
Recall the following lemma from \cite{Be02}.

\begin{Lemma}
Let $F_i$ be an $i$-dimensional face of the Dirichlet-Voronoi cell of $\p + \BB^d$, in a unit ball packing in $\Eu^d$. Let the distance
of $\aff F_i$ from $\p$ be $R < \sqrt{2}$. If $F_{i-1}$ is an $(i-1)$-dimensional face of $F_i$, then the distance of $\aff F_{i-1}$ from $\p$
is at least $\frac{2}{\sqrt{4-R^2}}$.
\end{Lemma}

This immediately yields that the distance of $o$ from any sideline of $F$ is at least $\frac{2}{\sqrt{4-x^2}}$, and from any vertex of $F$
at least $\sqrt{\frac{4-x^2}{3-x^2}}$.
By setting $H = \aff F$ and denoting the projection of $\o$ onto $H$ by $\c$, we may rephrase this observation in the following way:
$F$ is a polygon in $H$, containing the circle $C_1$ with centre $\c$ and radius $\sqrt{\frac{4}{4-x^2}-x^2} = \frac{2-x^2}{\sqrt{4-x^2}}$,
such that each vertex of $H$ is outside the circle $C_1$, with centre $c$ and radius $\frac{2-x^2}{\sqrt{3-x^2}}$.
Observe that we have a similar condition for the projection of $F$ onto the sphere $\S^2$.
Thus, to minimize $\vol_3(B_F)$, or equivalently, $\Svol_2(S_F) = 3 \vol_3(B_F)$, we may apply the following lemma from \cite{M65}.

\begin{Lemma}[Haj\'os]\label{lem:sphericalHajos}
Let $0 < r < R < \frac{\pi}{2}$, and let $C_r$ and $C_r$ be two concentric circles on the sphere $\S^2$, of radii $r$ and $R$, respectively.
let $\P$ denote the family of convex spherical polygons containing $C_r$, with no vertex contained in the interior of $C_2$.
If $P \in \P$ has minimal spherical area over all the elements of $\P$, then each vertex of $P$ lies on $C_R$, and each but at most one edge of $P$
touches $C_r$.
\end{Lemma}

Such a polygon is called a \emph{Haj\'os polygon} of the two circles.
By Lemma~\ref{lem:sphericalHajos}, we may assume that $F$ is a Haj\'os polygon, and compute $ \Svol_2(S_F)=3 \vol_3(B_F)$ under this condition.

Let $[\p,\q]$ be an edge of $H$ that touches $C_1$, and let $\mathbf{m}$ be the midpoint of $[\p,\q]$.
Let the angles of the triangle $T=\conv\{ \p,\mathbf{m},\c\}$, at $\p, \mathbf{m}$ $\c$, be $\beta$, $\gamma = \frac{\pi}{2}$ and $\alpha$, respectively.
Let $T'$ be the central projection of $T$ onto $\S^2$ from $o$, and denote the angles of $T'$ by $\alpha',\beta',\gamma'$, according to the notation in $T$.
We compute $\Svol_2(T')=\alpha'+\beta'+\gamma'-\pi$.
First, we observe that, by the properties of the projection, we have $\alpha'=\alpha$, and $\gamma'=\gamma=\frac{\pi}{2}$.
Since $\|\p-\c\| = \frac{2-x^2}{\sqrt{3-x^2}}$ and $\|\mathbf{m}-\c\| = \frac{2-x^2}{\sqrt{4-x^2}}$, an elementary computation yields $\|\p-\mathbf{m}\| = \frac{2-x^2}{\sqrt{(3-x^2)(4-x^2)}}$,
and
\[
\alpha'=\arctan \frac{1}{\sqrt{3-x^2}}.
\]

In the following, we use Lemma~\ref{lem:spherical}.

\begin{Lemma}\label{lem:spherical}
Let $H$ denote the tangent plane of the unit sphere $\S^2$ at some point $\p \in \S^2$.
Let $T= \conv \{ \p_1, \p_2, \p_3\}$ with $\p_1=\p$.
For $i=1,2,3$, let $\phi_i$ be the angle of $T$ at $\p_i$, and $\p'_i$ be the central projection of $\p_i$ on $\S^2$ from $o$.
Furthermore, let $T'$ be the central projection of $T$, with $\p'_i$ and $\phi'_i$ and $d'_i$ being the projections of $p_i$ and $\phi'_i$,
and the spherical length of the side of $T'$ opposite of $\p'_i$, respectively.
Then
\[
\tan \phi_2 = \tan \phi'_2 \cos d'_3, \quad \textrm{and} \quad \tan \phi_3 = \tan \phi'_3 \cos d'_2 .
\]
\end{Lemma}

\begin{proof}
Let $q$ be the orthogonal projection of $p_1$ onto the line containing $p_2$ and $p_3$, and let $q'$ be the central projection of $q$ onto $\S^2$.
Observe that the spherical angle $p'_1q'p'_2)\angle$ is a right angle. Thus, from the spherical law of cosine for angles, it follows that
\[
1=\tan(q'p'_1p'_2 \angle) \tan \phi'_2 \cos d'_3.
\]
Now, we have $q'p'_1p'_2 \angle = qp_1p_2 \angle = \frac{\pi}{2}-\phi_2$, from which the first equality readily follows.
The second one can be proven in a similar way.
\end{proof}

From Lemma~\ref{lem:spherical}, we readily obtain that $\tan \beta = \tan \beta' \cos \arctan \frac{\|\p-\c\|}{x}$, which yields
\[
\beta' = \arctan \frac{\sqrt{4-x^2}}{x}.
\]
Thus,
\begin{equation}\label{eq:sphericalarea}
\Svol_2(T') = \arctan \frac{1}{\sqrt{3-x^2}} + \arctan \frac{\sqrt{4-x^2}}{x} - \frac{\pi}{2}.
\end{equation}

Now, if $1 \leq x \leq \frac{2}{\sqrt{3}}$, then $\frac{\pi}{6} < \phi_0 \leq \alpha' \leq 0.659058 < \frac{\pi}{4}$.
Thus, $F$ has either five or six edges, depending on the values of $x$.
More specifically, if $1 \leq x < \sqrt{\frac{10-2\sqrt{5}}{5}} = 1.051462\ldots$, then $F$ has six, and otherwise five edges.
Using this, $\vol_3(B_F) = \frac{1}{3} \Svol_2 (S_F)$ can be computed similarly to $\Svol_2(T')$, which yields that if $1 \leq x \leq \sqrt{\frac{10-2\sqrt{5}}{5}}$, then
\[
\vol_3(B_F) = \frac{10}{3} \arctan\frac{\sqrt{4-x^2}}{x}- \frac{2}{3} \arccot \frac{ x \sqrt{3-x^2} \tan\left( 5 \arctan\left(\frac{1}{\sqrt{3-x^2}} \right) \right)}{\sqrt{4-x^2}}-\frac{2}{3} \pi .
\]
Let us denote the expression on the right by $f(x)$.
We may observe that if $\sqrt{\frac{10-2\sqrt{5}}{5}} < x \leq \frac{2}{\sqrt{3}}$, then the area of the sixth triangle appears with a negative sign
in $f(x)$, which yields, using a geometric observation, that in this case $\vol_3(B_F) > f(x)$.

Let
\[
F(x,\ll) = f(x) - C \pi \left( \frac{2}{3} \ll^3 - \ll^2 x + \frac{1}{3} x^3 \right),
\]
where $C = \frac{f(1)}{\pi \left( \frac{2}{3} \ll^3 - \ll^2 + \frac{1}{3} \right)} $.
Note that $F(1,\ll)=0$ for every value of $\ll$.
Thus, by (\ref{eq:Rogers1}), (\ref{eq:DVcellRogers}) and the inequality $\vol_3(B_F) \geq f(x)$, it follows that to prove Theorem~\ref{Rogers-type},
it is sufficient to show that $F(x,\lambda) \geq 0$ for every $1 \leq \ll <\frac{2}{\sqrt{3}}$ and $1 \leq x \leq \ll$.
On the other hand, it is an elementary exercise to check that $\frac{\partial^2 F}{\partial x^2} < 0$ on this region, which yields that
$F(x,\ll)$ is minimal at $F(1,\ll)$ or $F(\ll,\ll)$.
We may observe that $F(\ll,\ll) = f(\ll)$ is greater than four times the value of the expression in (\ref{eq:sphericalarea}) at $x = \ll$, which is positive.
Thus, $F(x,\ll)$ is not negative on the examined region, from which Theorem~\ref{Rogers-type} follows.

\section{Proof of Theorem~\ref{Rogers-type2}}
\bigskip

The proof is based on a somewhat modified version of the proof of Rogers' simplex packing bound, as described in \cite{R64}.

Recall that $\ll= 1+\lambda$, and
let $\P$ be a unit ball packing in $\Eu^d$ and let $V$ be the Voronoi cell of some ball in $\P$, say $\BB^d$.
Without loss of generality, we may assume that $\P$ is \emph{saturated}, e.g. there is no room to add additional balls to it;
this implies, in particular, that $V$ is bounded.

We partition $V$ into simplices in the following way.
Let $\c_0 = \o$.
Consider any sequence $F_{d-1} \supset F_{d-2} \supset \ldots \supset F_1\supset F_0$ of faces of $V$ such that $\dim F_{d-i} = d-i$ for every $1\le i\le d$.
Let $\c_i$ be the point of $F_{d-i}$ closest to the origin $\o$.
By induction on dimension, one can see that the simplices of the form $S=\conv \{ \c_0,\c_1,\ldots,\c_d\}$, some of which might be degenerate,
indeed tile $V$.
These simplices are called \emph{Rogers simplices}.
In the following we consider such a simplex $S = \conv \{ \c_0,\c_1,\ldots,\c_d\}$, with the indices chosen as in the previous sentence,
and recall one of their well-known properties (cf. p. 80, Lemma 2, \cite{R64}).

\begin{Lemma}[Rogers]\label{lem:Rogers}
For any Rogers simplex $S = \conv \{ \c_0,\c_1,\ldots,\c_d\}$ and for any $1 \leq i \leq j \leq d$, we have
\[
\langle \c_i , \c_j \rangle \leq \frac{2i}{i+1} .
\]
\end{Lemma}

Now, consider the $d$-dimensional simplex in $\Eu^{d+1}$, with vertices $(0,\ldots, 0, \sqrt{2}, 0, \ldots, 0)$, where the $i$th coordinate is $\sqrt{2}$,
and $i=1,2,\ldots,d+1$. This simplex is regular, and has edge-length two.
A barycentric subdivision divides this simplex into $(d+1)!$ congruent $d$-dimensional orthoschemes, one of which
has vertices $\q_i = \left( \frac{\sqrt{2}}{i+1}, \ldots, \frac{\sqrt{2}}{i+1}, 0, \ldots , 0  \right)$, where $\q_i$ has $i+1$ nonzero coordinates, and $i=0,1,\ldots,d$.
Set $Q = \conv \{ \q_0,\q_1,\ldots, \q_d \}$.
Then, for every $1 \leq i \leq j \leq d$, we have
\[
\langle \q_i-\q_0, \q_j-\q_0 \rangle = \langle \q_i,\q_j \rangle - \langle \q_i+\q_j, \q_0 \rangle + \langle \q_0, \q_0 \rangle
\]
\[
=(i+1) \frac{2}{(i+1)(j+1)} - \sqrt{2} \left( \frac{\sqrt{2}}{i+1} +  \frac{\sqrt{2}}{j+1} \right) + 2 = \frac{2i}{i+1} \leq \langle \c_i, \c_j \rangle .
\]

Let $A : \Eu^d \to \Eu^{d+1}$ be the affine map satisfying $A(\c_i) = \q_i$ for $i=0,1,\ldots,d$.
Consider any $\p \in S$, with $\|\p-\c_0\|=\|\p\| \leq \ll$.
Then $\p$ is the convex combination of the vertices of $S$; that is,
$\p = \sum_{i=0}^{d} \alpha_i \c_i$, where $\sum_{i=0}^d \alpha_i = 1$, and $\alpha_i \geq 0$ for every value of $i$.
By Lemma~\ref{lem:Rogers}, we have
\[
\langle A(\p) -\q_0, A(\p)-\q_0 \rangle = \left\langle \left( \sum_{i=0}^d \alpha_i \q_i \right) - \q_0, \left( \sum_{j=0}^d \alpha_j \q_j \right) - \q_0 \right\rangle
\]
\[
=\sum_{i=1}^d \sum_{j=1}^d \alpha_i \alpha_j \langle \q_i -\q_0,\q_j-\q_0 \rangle
\leq \sum_{i=1}^d \sum_{j=1}^d \alpha_i \alpha_j \langle \c_i, \c_j \rangle = \langle \p, \p \rangle.
\]
Hence, $A\left( (S \cap \left( \ll  \BB^d \right) \right) \subseteq Q \cap \left(\q_0 +  \ll \BB^{d+1}\right)$.

On the other hand, affine maps preserve volume ratio, and thus,
\[
\frac{\vol_d\left(S \cap \ll  \BB^d \right)}{\vol_d(S)} \leq \frac{\vol_d \left( Q \cap \left( \q_0 + \ll \BB^{d+1}  \right) \right) }{\vol_d(Q)},
\]
which readily yields the assertion.

\section{Proof of Theorem~\ref{Rogers-type3}}

Consider a unit ball packing $\P$ in $\Eu^3$, and let $V$ be the Voronoi cell of some ball of $\P$, say $\BB^3$.
Let $F$ be a face of $V$, and denote the intersection of the conic hull of $F$ with $V$, $\BB^3$ and $\bd \BB^3 = \S^2$ by $V_F$, $B_F$ and $S_F$, respectively.
Furthermore, we set $V'_F := V_F \cap \left( \ll \BB^3) \right)$ with $\ll= 1+\lambda$.
In the proof, we examine the quantity $\frac{\vol_3(V_F')}{\vol_3(V_F)}$.
Without loss of generality, we may assume that $F$ contains the intersection of $\ll \BB^3$ and $\aff F$.

Let the distance of $F$ and $\o$ be $x$, where $1 \leq x \leq \ll < \frac{2}{\sqrt{3}}$.
An elementary computation yields that $\vol_3\left( \left( \ll B_F \right) \setminus V_F \right) =
\pi \left( \frac{2}{3} \ll^3 - \ll^2 x + \frac{1}{3} x^3 \right)$, from which it follows that
\begin{equation}\label{eq:DVcellRogers3}
\vol_3(V'_F) = \ll^3 \vol_3(B_F) - \pi \left( \frac{2}{3} \ll^3 - \ll^2 x + \frac{1}{3} x^3 \right).
\end{equation}

We introduce a spherical coordinate system on $\S^2$, with the polar angle $\theta \in [0,\pi]$ measured from the North Pole, which we define
as the point of $\S^2$ closest to $F$.
Now, we define the functions $f(\theta)$, $g(\theta)$ and $h(\theta)$, as the volume of the set of points of $V'_F$, $V_F$ and $B_F$, respectively,
with polar angle at most $\theta$.
We observe that the proof of Sublemma 5 in \cite{B00} yields that $\frac{h(\theta)}{g(\theta)}$ is a decreasing function of $\theta$; or even more, that for any fixed value of $\bar{\theta}$ and variable $\theta \geq \bar{\theta}$,
the function $\frac{h(\theta)-h(\bar{\theta})}{g(\theta)-g(\bar{\theta})}$ decreases.

Let $v_0 = f(\theta_0) = g(\theta_0)$, where $\theta_0$ is the largest $\theta$ with $f(\theta) = g(\theta)$,
and observe that for any $\theta \geq \theta_0$, we have $f(\theta)-f(\theta_0) = \ll^3 (h(\theta)-h(\theta_0))$.
Since $f$, $g$ and $h$ are increasing functions and  $\frac{h(\theta)-h(\theta_0)}{g(\theta)-g(\theta_0)}$ decreases, we have that
$\frac{h'(\theta)}{g'(\theta)} \leq \frac{h(\theta)-h(\theta_0)}{g(\theta)-g(\theta_0)}$.
As $f'(\theta) =\bar{\lambda^3} h'(\theta)$ for every $\theta > \theta_0$, this yields that
\[
\frac{f'(\theta)}{g'(\theta)} \leq \frac{f(\theta)-f(\theta_0)}{g(\theta)-g(\theta_0)} = \frac{f(\theta)-v_0}{g(\theta)-v_0},
\]
from which it follows that
\[
\left( \frac{f}{g} \right)' =  \frac{ \left( f'(g-v_0)-g'(f-v_0) \right) + (f'-g')v_0}{g^2} \leq \frac{(f'-g')v_0}{g^2}.
\]
On the other hand, it is easy to see that $f'(\theta) \leq g'(\theta)$ for every $\theta \geq \theta_0$ and therefore $\left( \frac{f}{g} \right)'  \leq 0$.

Let $\c$ be the closest point of $\aff F$ to $\o$. It is a well-known fact \cite{B00} that the vertices of $F$ are not in the interior of the circle
$G$, with centre $\c$ and radius $\sqrt{\frac{3}{2}-x^2}$, and that it contains the circle $G_0$, with centre $\c$ and radius $\frac{2-x^2}{\sqrt{4-x^2}}$.
Furthermore, at most five sides of $F$ intersect the relative interior of $G$ (cf. \cite{B00} or \cite{M93}).
Let us define $V_{F \cap G}$ and $V'_{F \cap G}$ analogously to $V_F$ and $V'_F$ respectively.
Since $\frac{f(\theta)}{g(\theta)}$ decreases and for $\theta$ sufficiently close to $\frac{\pi}{2}$, it is equal to $\frac{\vol_3(V'_F)}{\vol_3(V_F)}$, we have
$\frac{\vol_3(V'_F)}{\vol_3(V_F)} \leq \frac{\vol_3(V'_{F \cap G})}{\vol_3(V_{F\cap G})}$.
Let $M_0$ be a regular pentagon, with centre $\c$, such that the spherical area of the projection of $M_0 \cap G$ onto $\S^2$ is equal to that of $F \cap G$.
Then, using the idea of Proposition 1 from \cite{B00}, we obtain that
\[
\frac{\vol_3(V'_{F \cap G})}{\vol_3(V_{F\cap G})} \leq \frac{\vol_3(V'_{M_0 \cap G})}{\vol_3(V_{M_0\cap G})}.
\]

Now, assume that the distance of the sides of $M_0$ from $G_0$ is $y$. Then $\frac{2-x^2}{\sqrt{4-x^2}} \leq y \leq \sqrt{\frac{3}{2}-x^2}$.
Let $M(y)$ denote the regular pentagon, with $\c$ as its centre and its sidelines being at distance $y$ from $c$.
We show that the relative density of $M(y)$ is a decresing function of $y$.
Let $y_1$ arbitrary. Let $\mathbf{x}_1$ be the midpoint of a side of $M(y_1)$, $\mathbf{v}_1$ a vertex of this side, $\mathbf{y}_1$ and $\mathbf{z}_1$ the intersections of $[\mathbf{x}_1,\mathbf{v}_1]$ and $[\c,\mathbf{v}_1]$, respectively, with the relative boundary of $G$.
Let $X_1 = \conv \{ \mathbf{x}_1, \mathbf{y}_1, \c \}$ and $U_1$ the convex hull of $\c$ and the shorter circle arc in $\relbd G$, connecting $\mathbf{y}_1$ and $\mathbf{z}_1$. Let $V_{X_1}$ and $V'_{X_1}$ be the set of points of the conic hull of $X_1$ in $V_F$, and in $V'_F$, respectively.
We define $V_{U_1}$ and $V'_{U_1}$ similarly.
Now we set $y_2 > y_1$, and introduce the same points and sets with index $2$ in the same way for $M(y_2)$.

Consider the sets $\{ \c, \mathbf{x_1}, \mathbf{y}_1 \}$ and $\{ \c, \mathbf{x_2}, \mathbf{y}_2 \}$. Observe that the inner product of any two vectors from the first set is at least as large as that of the corresponding two vectors from the second set. Thus, Rogers' method, as described in the previous section, yields that
\[
\frac{\vol_3(V'_{X_1})}{\vol_3(V_{X_1})} \geq \frac{\vol_3(V'_{X_2})}{\vol_3(V_{X_2})} \geq \frac{\vol_3(V'_{U_1})}{\vol_3(V_{U_1})} = \frac{\vol_3(V'_{U_2})}{\vol_3(V_{U_2})}.
\]
Observe that $\frac{\vol_3(V_{U_1})}{\vol_3(X_1)} \leq \frac{\vol_3(V_{U_2})}{\vol_3(X_2)}$.
Then an algebraic transformation like in the proof of Proposition 2 of \cite{B00} yields that the relative density of $M(y_1)$ is greater than or equal to that of $M(y_2)$. Thus, we may assume that $M$ is a regular pentagon, circumscribed about $G_0$.

It remains to show that if $1 \leq \ll \leq \frac{2}{\sqrt{3}}$, then the relative density for this pentagon is maximal on $x \in [1,\ll]$ if $x=1$. To do this, we apply an argument similar to the one in Section~\ref{sec:proofRogers-type1}.
Consider the function $f(x,\ll) = \vol_3(V'_M)-C \vol_3(V_M)$, where $C=C(\ll)$ is the value of the relative density of $M$, at $x=1$.
We may compute $f(x,\ll)$ using elementary methods, and the tools described in Section~\ref{sec:proofRogers-type1}.
An elementary calculation shows that if $\ll \leq 1.14$ nor $x \leq 1.12$, then $\frac{\partial}{\partial x} f(x,\ll) < 0$, whereas
if $1.14 \leq \ll \leq \frac{2}{\sqrt{3}}$ and $1.12 \leq x \leq \ll$, then $f(x,\ll) < 0$.
Thus, for every $1 \leq \ll \leq \frac{2}{\sqrt{3}}$, the relative density of $M$ is maximal if $x=1$, which yields that
\[
\bar{\delta}_3(\lambda) \leq \frac{ \left( 20 \sqrt{6} \phi_0 -4 \sqrt{6} \pi -10\pi \right)\ll^3 +18 \pi \ll^2 -6\pi}{3\pi-15\phi_0 + 5\sqrt{2}} .
\]

\section{Concluding Remarks}

We note that Theorem~\ref{thm:Groemer} immediately yields the following.

\begin{Corollary}
Let $\P^n = \{\c_i+\BB^2 : i=1,2,\ldots, n \}$ be a packing of unit disks in $\Eu^2$. Let $1 < \ll= 1+\lambda < \frac{2}{\sqrt{3}}$.
Set $A = \area \bigcup_{i=1}^n \left( \c_i+\ll \BB^2 \right)$, and $P= \perim \conv \left( \bigcup_{i=1}^n \left( \c_i+ \ll \BB^2 \right) \right)$. Then
\[
A \geq \left( \ll^2 \left( \pi - 6\arccos \frac{1}{\ll} \right) + 6\sqrt{\ll^2 - 1} \right) n + \left( \ll^2 \arccos\frac{1}{\ll} - \sqrt{\ll^2-1} \right) P +  \ll^2 \pi .
\]
\end{Corollary}

Furthermore, the proof of Theorem~\ref{thm:Groemer} can be modified in a straightforward way to prove Remark~\ref{rem:Groemer} below.
In this remark, we use the notations of Definitions~\ref{defn:lambdasimple} and \ref{defn:perimeter}, and set $\bar{\Lambda}: = 2.926949 \ldots$ for
the smallest root of the equation
\[
\left(\sqrt{3}- \frac{\ll^2 \pi}{2} \right) (\ll-1) - \ll \sqrt{\ll^2-1}+ \ll^3 \arccos \frac{1}{\ll} = 0
\]
that is greater than one.

\begin{Remark}\label{rem:Groemer}
Let $\P^n = \left\{ \c_i+\BB^2 : i=1,2,\ldots,n \right\}$ be a packing of $n$ unit disks in $\Eu^2$, and let $\frac{2}{\sqrt{3}} \leq \ll \leq \bar{\Lambda}$.
Let $A = \area \left( \bigcup_{i=1}^{n}(\c_i+\ll\BB^2) \right)$, and $P=\perim (G(\P^n))$.
Then
\begin{equation}\label{eq:Groemer2}
A \geq \sqrt{12} n + \frac{1}{2} \left( \ll^2 \left( \frac{\pi}{2} - \arccos \frac{1}{\ll} \right) + \sqrt{\ll^2-1}-\sqrt{3} \right) P + \ll^2 \pi .
\end{equation}
\end{Remark}

We note that both in Theorem~\ref{thm:Groemer} and Remark~\ref{rem:Groemer}, equality occurs, for example, if $\P^n$ is a subfamily of the densest lattice packing of the plane $\Eu^2$ with unit disks. Nevertheless, this is not true if $\lambda$ is sufficiently large.

\begin{Remark}\label{rem:Schurmann}
Let $n \geq 371$.
Then there exists $\lambda_0=\lambda_0(n)$ such that for any $\lambda > \lambda_0$,
if for some packing $\P^n:=\{\c_i+\BB^2 : i=1,2,\ldots,n \}$, $\area \left( \bigcup_{i=1}^n \left( \c_i+(1+\lambda) \BB^2 \right) \right)$
is minimal over all packings of $n$ unit disks in $\Eu^2$, then $\P^n$ is not a subfamily of the densest lattice packing of unit disks in $\Eu^2$.
\end{Remark}

\begin{proof}
For any $\lambda> 0$ and unit disk packing $\P^n:=\{\c_i+\BB^2 : i=1,2,\ldots,n \}$, we set
\[
f(\lambda,\P^n) = \area \left( \bigcup_{i=1}^n \left( \c_i+(1+\lambda) \BB^2 \right) \right),
\]
and $\ll = 1+\lambda$.

Recall the following $d$-dimensional result of Capoyleas and Pach (\cite{CaPa}) and Gorbovickis (\cite{Go}) stating that
$f(\lambda,\P^n)$, as a function of $\lambda$, is analytic in some punctured neighbourhood of infinity,
has a pole of order $d$ at infinity, and, in particular,
\[
f(\lambda,\P^n) = \omega_d \ll^d + + M_d[\conv\{ \c_1, \c_2, \ldots, \c_n\}] \ll^{d-1} + g(\lambda,\P^n),
\]
where
\[
M_d(K) = \int_{\S^{d-1}} \max \left\{ \langle x,u \rangle : x \in K \right\} d \sigma(u)
\]
is the $d$-dimensional mean width of the convex body $K$ (up to multiplication by a dimensional constant), and
$\lim_{\lambda \to \infty} \frac{g(\lambda,\P^n)}{\ll^{d-1}} = 0$.

Note that if $K$ is a $2$-dimensional convex body, then $M_2(K) = \frac{\perim(K)}{\pi}$, where $\perim(K)$ is the perimeter of $K$.
Thus, in the planar case, we have
\[
f(\lambda,\P^n) = \pi \ll^2 + C \perim\left( \conv\{ \c_1, \c_2, \ldots, \c_n\} \right) \ll + g(\lambda,\P^n),
\]
where $\lim_{\lambda \to \infty} \frac{g(\lambda,\P^n)}{\lambda} = 0$, and $C > 0$.

Let $n \geq 371$ fixed. Then, by a result of Sch\"urmann \cite{Sch02}, only non-lattice packings are the extremal packings under the perimeter function.
Let $P$ and $P_{lattice}$ the minimum of $\perim\left( \conv\{ \c_1, \c_2, \ldots, \c_n\} \right)$ over  the family of $n$ element packings and lattice packings, respectively, and let $x = \frac{P_{lattice}-P}{C}> 0$.

By compactness, there is some $\lambda_0 = \lambda_0(n)$ such that for any $\lambda > \lambda_0$ and for any $n$-element packing $\P^n$, we have
$g(\lambda,\P^n) \leq \frac{x}{3} \ll$.
Then, if $\P^n$ minimizes $\perim\left( \conv\{ \c_1, \c_2, \ldots, \c_n\} \right)$ over the family of all $n$ element packings, and $\P^n_{lattice}$ is any lattice packing, we have
\[
f(\lambda,\P^n_{lattice}) \geq \pi \ll^2 + C P_{lattice} \ll - \frac{x}{3} C \ll > \pi \ll^2 + C P \ll + \frac{x}{3} C \ll \geq f(\lambda,\P^n).
\]
\end{proof}

\begin{Problem}
Does Remark~\ref{rem:Schurmann} hold with some universal $\lambda_0$ independent of $n$?
\end{Problem}

\begin{Remark}
Let $\P^n$ be a packing on $n$ unit disks in $\Eu^2$, $0 < \lambda < \frac{2}{\sqrt{3}}-1$, and assume that the boundary of the $\lambda$-intersection
graph $G(\P^n)$ of $P^n$ is connected (cf. Definition~\ref{defn:lambdasimple}), and contains no edge of $\P^n$ more than once. Then we have equality in (\ref{eq:Groemer}) of Theorem~\ref{thm:Groemer} if, and only if $\P^n$ is a subfamily of the densest lattice packing of unit disks.
\end{Remark}

\begin{Definition}
Let $\mathbf{D} \subset \Eu^3$ be a regular dodecahedron circumscribed the unit ball $\BB^3$.
Then, for $ \lambda$ with $0<\lambda<\sqrt{3}\tan \frac{\pi}{5}=1.258408\dots$, we set
\[
\tau_3(\lambda):= \frac{\vol_d (\BB^3)}{\vol_d ( \mathbf{D}\cap(1+\lambda)\BB^3 )},
\]
and
\[
\bar{\tau}_3(\lambda): = \frac{\vol_d (\mathbf{D}\cap (1+\lambda)\BB^3 )}{\vol_d (\mathbf{D})},
\]
where $\sqrt{3}\tan \frac{\pi}{5}$ is the circumradius of $\mathbf{D}$.
\end{Definition}

Based on the proof of the dodecahedral conjecture of Hales and McLaughlin (\cite{HaMc}), it seems reasonable to ask the following.

\begin{Problem}
Is it true that for every $0<\lambda<\sqrt{3}\tan \frac{\pi}{5}=1.258408\dots$, we have
$\delta_3(\lambda) \leq \tau_3(\lambda)$ and $\bar{\delta}_3(\lambda) \leq \bar{\tau}_3(\lambda)$?
\end{Problem}

We say that a packing $\P$ of unit balls in $\Eu^{d}$ is \emph{universally optimal} if $\bar{\delta}_d(\mathbf{P}_{\lambda})=\bar{\delta}_d(\lambda)$ holds for all $\lambda\ge 0$.
(We note that this notion is a Euclidean analogue of the notion of perfectly distributed points on a sphere introduced by L. Fejes T\'oth in \cite{F71}, which is different from the notion of universally optimal distribution of points on spheres introduced by Cohn and Kumar in \cite{CoKu07}.)
Recall that $\mu_d>0$ is called the {\it simultaneous packing and covering constant} of the closed unit ball $\BB^d=\{ \mathbf{x}\in \Eu^{d}\ |\ \|\mathbf{x}\|\le 1\}$ in $\Eu^{d}$ if the following holds (for more details see for example, \cite{S09}): $\mu_d>0$ is the smallest positive real number $\mu$ such that there is a unit ball packing $\P:=\{\c_i+\BB^d\ |\ i=1,2,\dots \ {\rm with}\ \|\c_j-\c_k\|\ge 2\ {\rm for}\ {\rm all}\ 1\le j< k\}$ in $\Eu^{d}$ satisfying $\Eu^{d}=\bigcup_{i=1}^{+\infty}(\c_i+\mu\BB^d)$. Now, if $\P$ is a universally optimal packing of unit balls, then clearly $\bar{\delta}_d(\mathbf{P}_{\mu_d-1})=\bar{\delta}_d(\mu_d-1)=1$ and $\bar{\delta}_d(\mathbf{P}_{0})=\bar{\delta}_d(0)=\delta_d$. Next, recall that according to the celebrated result of Hales (\cite{H12}) $\delta_3=\frac{\pi}{\sqrt{18}}=0.740480\dots$ and it is attained by the proper face-centered cubic lattice packing of unit balls in $\Eu^{3}$. Furthermore, according to a theorem of B\"or\"oczky \cite{Bo2} $\mu_3=\sqrt{\frac{5}{3}}=1.290994\dots$ and it is attained by the proper body-centered cubic lattice packing of unit balls in $\Eu^{3}$. As a result it is not hard to see that {\it there is no universally optimal packing of unit balls in $\Eu^{3}$}. One may wonder whether there are universally optimal packings of unit balls in $\Eu^{d}$ for $d\ge4$?

To state our observation about the planar case, we introduce the notion of \emph{uniformly recurrent packings}, defined in \cite{GKup}.
First, we generalize the notion of Hausdorff distance $d(\cdot,\cdot)$ of two convex bodies.
For two packings $\P_1$, $\P_2$ of convex bodies, we say that $d(\P_1,\P_2) \leq \varepsilon$, if for any ${\bf K}_1 \in \P_1$ contained in the unit ball of radius $\frac{1}{\varepsilon}$ with the origin as its center, there is a unique ${\bf K}_2 \in \P_2$ such that $d({\bf K}_1,{\bf K}_2) \leq \varepsilon$, and vice versa.
We say that $\P_1$ is a \emph{limit} of $\P_2$, denoted as $\P_1 \succeq \P_2$, if a sequence of translates of $\P_2$ converges to $\P_1$, in the topology defined by Hausdorff distance.
A packing is \emph{uniformly recurrent}, if it is maximal in the weak partial order $\succeq$ of the family of packings.
Kuperberg \cite{GKup} proved that the only uniformly recurrent densest packing of Euclidean unit disks is the densest hexagonal lattice packing;
his proof was based on the observation that the only minimal area Voronoi cell of a unit disk is the regular hexagon circumscribed about the disk.
Since this observation holds for any planar packing of soft disks, using a slight modification of the proof of Theorem~\ref{Rogers-type2} in \cite{GKup}, we have the following.

\begin{Remark}\label{perfect 2D packings}
Let $\P$ be a packing of unit disks in the Euclidean plane.
Then the following are equivalent.
\begin{enumerate}
\item $\P$ is the densest hexagonal lattice packing.
\item $\P$ is uniformly recurrent and universally optimal.
\item $\P$ is uniformly recurrent, and $\bar{\delta}_2(\mathbf{P}_{\lambda}) = \bar{\delta}_2(\lambda)$ for some $0 \leq \lambda < \frac{2}{\sqrt{3}}-1$.
\end{enumerate}
\end{Remark}

\vspace{1cm}

\medskip

\noindent
K\'aroly Bezdek
\newline
Department of Mathematics and Statistics, University of Calgary, Calgary, Canada,
\newline
Department of Mathematics, University of Pannonia, Veszpr\'em, Hungary,
\newline
{\sf E-mail: bezdek@math.ucalgary.ca}
\newline
and
\newline
\noindent
Zsolt L\'angi
\newline
Department of Geometry, Budapest University of Technology, Budapest, Hungary
\newline
{\sf E-mail: zlangi@math.bme.hu}

\end{document}